\documentclass[12pt,reqno]{amsart}

\usepackage{times}
\usepackage[T1]{fontenc}
\usepackage{mathrsfs}
\usepackage{latexsym}
\usepackage[dvips]{graphics}
\usepackage{epsfig}
\usepackage{amsmath,amsfonts,amsthm,amssymb,amscd}
\input amssym.def
\input amssym.tex
\usepackage{color}
\usepackage{hyperref}
\usepackage{url}

\newcommand{\burl}[1]{\textcolor{blue}{\url{#1}}}

\newcommand{\emaillink}[1]{\textcolor{blue}{\href{mailto:#1}{#1}}}

\newtheorem{thm}{Theorem}[section]

\newcommand\bd{\begin{displaymath}}
\newcommand\ed{\end{displaymath}}
\newcommand\be{\begin{equation}}
\newcommand\ee{\end{equation}}
\newcommand\bea{\begin{eqnarray}}
\newcommand\eea{\end{eqnarray}}
\newcommand\bi{\begin{itemize}}
\newcommand\ei{\end{itemize}}
\newcommand\ben{\begin{enumerate}}
\newcommand\een{\end{enumerate}}
\newcommand\bc{\begin{center}}
\newcommand\ec{\end{center}}
\newcommand\ba{\begin{array}}
\newcommand\ea{\end{array}}

\newcommand{\R}{\ensuremath{\Bbb{R}}}

\newcommand{\twocase}[5]{#1 \begin{cases} #2 & \text{#3}\\ #4
&\text{#5} \end{cases}   }

\newcommand{\E}{\ensuremath{\mathbb{E}}}

\numberwithin{equation}{section}

\begin{document}




\title[The Weibull Distribution and Benford's Law]{The Weibull Distribution and Benford's Law}

\author{Victoria Cuff}\email{\emaillink{vcuff@g.clemson.edu}}
\address{Department of Mathematics, Clemson University, Clemson, SC}

\author{Allison Lewis}\email{\emaillink{allewis2@ncsu.edu}}
\address{Department of Mathematics, North Carolina State University, Raleigh, NC 27695}

\author{Steven J. Miller}\email{\emaillink{sjm1@williams.edu}, \emaillink{Steven.Miller.MC.96@aya.yale.edu}}
\address{Department of Mathematics and Statistics, Williams College, Williamstown, MA 01267}

\bigskip
\begin{abstract}
Benford's law states that many data sets have a bias towards lower leading digits (about 30\% are 1s). There are numerous applications, from designing efficient computers to detecting tax, voter and image fraud. It's important to know which common probability distributions are almost Benford. We show the Weibull distribution, for many values of its parameters, is close to Benford's law, quantifying the deviations. As the Weibull distribution arises in many problems, especially survival analysis, our results provide additional arguments for the prevalence of Benford behavior. The proof is by Poisson summation, a powerful technique to attack such problems.
\end{abstract}



\date{\today}

\thanks{The first and second named authors were supported by NSF Grant DMS0850577 and Williams College; the third named author was supported by NSF Granta DMS0970067 and DMS1265673. This work was done at the 2010 SMALL REU at Williams College, and a summary  of it will appear in Chapter Three of \emph{The Theory and Applications of Benford's Law}, to be published by Princeton University Press and edited by the third named author. Since this work was written many of these results have been independently derived and applied them to Internet traffic; see the work of Arshadi and Jahangir \cite{AJ}. }

\maketitle


\section{Introduction to and Applications of Benford's Law}

For any positive number $x$ and base $B$, we can represent $x$ in scientific notation as $x=S_B(x)\cdot B^{k(x)}$, where $S_B(x) \in [1,B)$ is called the significand\footnote{The significand is sometimes called the mantissa; however, such usage is discouraged by the IEEE and others as mantissa is used for the fractional part of the logarithm, a quantity which is also important in studying Benford's law.} of $x$ and the integer $k(x)$ represents the exponent. Benford's Law of Leading Digits proposes a distribution for the significands which holds for many data sets, and states that the proportion of values beginning with digit $d$ is approximately
\bea
\mathrm{Prob}({\rm first\ digit\ is\ }d\ {\rm base\ }B) \  =\ \log_{B}\left(\frac{d+1}{d}\right);
\eea more generally, the proportion with significand at most $s$ base $B$ is \bea {\rm Prob}(1 \le S_B \le s) \ = \ \log_{B} s. \eea In particular, base 10 the probability that the first digit is a 1 is about 30.1\% (and not the 11\% one would expect if each digit from 1 to 9 were equally likely).

This leading digit irregularity was first discovered by Newcomb \cite{Ne} in 1881 , who noticed that the earlier pages in the logarithmic books were more worn than other pages. Fifty years later Benford \cite{Ben} observed the same digit bias in a variety of data sets. Benford studied the distribution of the first digits of 20 sets of data with over 20,000 total observations, including river lengths, populations, and mathematical sequences. For a full history and description of the law, see \cite{Hi1,Rai}, or go to the Online Benford Bibliography \cite{BH4} for additional reading.

One of the most fascinating aspects of Benford's law is the large and diverse list of fields studying it (auditing, computer science, dynamical systems, engineering, number theory, and statistics, to list a few). There are numerous applications, especially in fraud and data integrity. Two of the more famous are detecting tax and voter fraud (Cho and Gaines \cite{CG07}, Mebane \cite{Me}, Nigrini \cite{Nig1,Nig2}), but there are also applications in many other fields, ranging from round-off errors in computer science (Knuth \cite{Knu}) to detecting image fraud and compression in engineering \cite{AHMP-GQ}. Already Benford's law has led to a variety of tests, either to detect fraud (in everything from corporate returns to medical studies) or to test data integrity; see for example \cite{JS,Nig2,NiMi}.

In the next section we discuss attempts to explain the prevalence of Benford's law; unfortunately, some of these approaches are flawed, and have been incorrectly used for decades. Our purpose in this article is to highlight techniques from Fourier analysis that may not be widely known to the diverse group of researchers and aficionados in the field, emphasizing how Poisson summation provides a clean and correct way to quantify deviations from Benford's law for a variety of phenomena. Our main result is to quantify how close Weibull distributions are to Benford (we state these in Theorem \ref{thm:main} in \S\ref{sec:mainresults}, after first reviewing the needed pre-requisites in \S\ref{sec:prereqs}; the proof is given in \S\ref{sec:approachviaderivs}). For certain values of the scale and shape parameter these distributions are almost Benford; this is quite important, as many survival distributions are modeled by Weibull distributions, and thus Benford tests are applicable.

\section{Explanations of Benford's Law}

There have been numerous attempts to pass from observing the prevalence of Benford's law to explaining its occurrence in different and diverse systems. Such knowledge gives us a deeper understanding of which natural data sets should follow Benford's law. One of the earliest and most popular is due to Feller  \cite{Fe}, and has been the subject of many articles and papers since (a very good, recent description of this approach is given in Fewster \cite{Few}). It suggests that Benford behavior arises when a probability distribution is spread out over several orders of magnitude. Unfortunately, while some distributions satisfying this condition are close to Benford, others are not, and the method is sadly fundamentally flawed. See \cite{BH1,BH2,Hi3} for detailed critiques of this method. The first rigorous explanation of Benford's law due to Hill \cite{Hi2} through scale invariance and measure theory (essentially, the distribution of leading digits should be invariant if we change scale); see also \cite{BH3}.

Rather than trying to prove why so many different phenomena are almost Benford, another approach is to study specific, important instances. In particular, there is an extensive literature on the leading digits of random variables and products of random variables of specific distributions (see for example \cite{MiNi1}). While these arguments cannot be as general, the systems described arise in many important applications, making the importance of these researches clear.

The starting point of this work is the paper by Leemis, Schmeiser, and Evans \cite{LSE}, who champion this viewpoint. They ran numerical simulations on a variety of parametric survival distributions to examine conformity to Benford's Law. Among these distributions was the Weibull distribution, whose density is  \be \label{eq:stndweibullintro} \twocase{f(x; \alpha, \gamma)\   =  \   }{\frac{\gamma}{\alpha} \left(\frac{x}{\alpha}\right)^{(\gamma - 1)} \exp\left( - \left( \frac{x}{\alpha} \right) ^\gamma \right)}{if $x  \ge 0$}{0}{otherwise,}\ee where $\alpha, \gamma >  0$. Note that $\alpha$ adjusts the scale of the data and only $\gamma$ affects the shape of the distribution.\footnote{One could introduce another parameter, $\beta$, which would represent a translation of the data. Doing so replaces $x$ with $x-\beta$, and the condition $x \ge 0$ becomes $x \ge \beta$. In this paper we concentrate on the case $\beta = 0$.} Special cases of the Weibull include the exponential distribution ($\gamma=1$) and the Rayleigh distribution ($\gamma=2$). The most common use of the Weibull is in survival analysis, where a random variable $X$ modeled by the Weibull represents the ``time-to-failure'', resulting in a distribution where the failure rate is modeled relative to a power of time.

The Weibull distribution arises in problems in such diverse fields as food contents, engineering, medical data, politics,  pollution and sabermetrics, along with many others; see \cite{An,Ca,CB,Fr,MABF,Mik,Mi,TKD,We,Yi,ZLYM} to name just a few. As the extensiveness of this list indicates, many data sets follow a Weibull distribution, and thus if we are going test for fraud or data integrity, it is essential to quantify how close these distributions are to Benford. Our goal in this work is to provide proofs of the observations of Leemis, Schmeiser, and Evans \cite{LSE} that Weibulls are often close to Benford, emphasizing the ideas behind the method as these are applicable to a variety of other problems (see for example \cite{JKKKM,KM,MiNi2}).

\section{Mathematical Preliminaries}\label{sec:prereqs}

Our analysis generalizes the work of \cite{MiNi2}, where the exponential case was studied in detail (see also \cite{DL} for another approach to analyzing exponential random variables). The main ingredients come from Fourier analysis, in particular applying Poisson summation to the derivative of the cumulative distribution function of the logarithms modulo 1, $F_B$. We first review some needed definitions, then describe why it is so useful to study the logarithms modulo 1, and conclude with a quick review of Poisson summation.\\

\ben


\item The Gamma function $\Gamma(s)$ generalizes the factorial function; we have $\Gamma(n+1) = n!$ for $n$ a non-negative integer, and for $\Re(s) > 0$ we have $$\Gamma(s) \ = \ \int_0^\infty e^{-x} x^{s-1}dx$$ (we will need to evaluate the Gamma function at complex arguments in our analysis); here $\Re(z)$ denotes the real part of $z$. See \cite{WW} for an introduction and proofs of needed properties.

\item We say $a$ is congruent to $b$ modulo 1 if $a-b$ is an integer; we denote this by $a=b \bmod 1$.

\item A sequence $\{a_n\}_{n=1}^\infty \subset [0,1]$ is equidistributed if \be \lim_{N\to \infty} \frac{ \# \{n : n \leq N,\ a_n\in [a,b]\} }{N} \ = \ b-a\nonumber\ee for all $[a,b]\subset [0,1]$. Similarly a continuous random variable on $[0,\infty)$ whose probability density function is $p$ is equidistributed modulo $1$ if \be \lim_{T\to\infty} \frac{\int_0^T \chi_{a,b}(x)p(x)dx}{\int_0^T p(x)dx} \ = \ b-a\nonumber \ee for any $[a,b] \subset [0,1]$, where $\chi_{a,b}(x) = 1$ for $x\bmod 1 \in [a,b]$ and $0$ otherwise.

\item If $f$ is an integrable function (so $\int_{-\infty}^\infty |f(x)|dx < \infty$) then its Fourier transform, denoted $\widehat{f}$, is given by $$\widehat{f}(y) \ = \ \int_{-\infty}^\infty f(x) e^{-2\pi i xy} dx, \ \ \ \ {\rm where}\ e^{i u}\ =\ \cos u + i \sin u.$$ Note if $X$ is a random variable with density $f$ then this is a rescaled version of its characteristic function, $\E[e^{itX}]$.

\item Let $\eta > 0$. We say $f$ decays like $x^{-(1+\eta)}$ if there are constants $x_0, C_\eta > 0$ such that $|f(x)| \le C_\eta |x|^{-(1+\eta)}$ for all $|x| > x_0$.

\een \ \\

One of the most common ways to prove a system is Benford is to show that its logarithms modulo 1 are equidistributed. We quickly sketch the proof of this equivalence; see \cite{Dia,MiNi2,MT-B} for details. If $y_n = \log_B x_n \bmod 1$ (thus $y_n$ is the fractional part of the logarithm of $x_n$), then the significands of $B^{y_n}$ and $x_n = B^{\log_B x_n}$ are equal, as these two numbers differ by a factor of $B^k$ for some integer $k$. If now $\{y_n\}$ is equidistributed modulo 1, then by definition for any $[a,b] \subset [0,1]$ we have $\lim_{N\to\infty} \#\{n \le N: y_n \in [a,b]\}/N = b-a$. Taking $[a,b] = [0,\log_B s]$ implies that as $N\to\infty$ the probability that $y_n \in [0,\log_B s]$ tends to $\log_B s$, which by exponentiating implies that the probability that the significand of $x_n$ is in $[1,s]$ tends to $\log_B s$, the Benford probability.

Given a random variable $X$, let $F_B$ denote the cumulative distribution function of $\log_B X$ $\bmod$ $1$. The above discussion shows that Benford's law is equivalent to $F_B(z)=z$, or our original random variable $X$ is Benford if $F'_B(z)=1$. This suggests that a natural way to investigate deviations from Benford behavior is to compare the deviation of $F'_B(z)$ from 1, which would represent a uniform distribution.

Fourier analysis is ideally suited for these computations. The reason is that in general one cannot throw away part of a mathematical expression and maintain equality. For example, $\sqrt{(x \bmod 1) + (y \bmod 1)}$ is neither equal to nor congruent modulo 1 to $\sqrt{x+y}$; however, $e^{2\pi i x}$ does equal $e^{2\pi i (x \bmod 1)}$. By using the complex exponentials, it is harmless to drop modulo 1 restrictions. As these restrictions naturally arise in investigating the first digit, it is natural to attack the problem with Fourier techniques.

The last ingredient we need is Poisson summation. We don't state it in its most general form, as the following weak version typically suffices for Benford investigations due to the smoothness of the underlying densities. See \cite{MT-B} or \cite{SS} for a proof.

\begin{thm}[Poisson summation]\label{thm:poissonsum} Let $f, f'$ and $f''$ be continuous functions which decay like $x^{-(1+\eta)}$ for some $\eta > 0$. Then $$\sum_{n=-\infty}^\infty f(n) \ = \ \sum_{n=-\infty}^\infty \widehat{f}(n).$$ \end{thm}

Our assumptions about $f$ imply that $\widehat{f}$ decays rapidly. The power of Poisson summation is that it typically allows us to exchange a slowly converging sum with a rapidly converging sum. In many applications only the $n=0$ term matters; if $f$ is a probability density then it integrates to 1 and hence $\widehat{f}(0) = 1$. For us, this is important as it implies a sum over non-zero $n$ can measure a deviation.

For example, consider the density of a normal random variable $Y$ with mean 0 and variance $N/2\pi$; this example is very important in showing Brownian motions and many product of independent random variables become Benford (see \cite{MT-B,MiNi1}). If we want to see how often $Y \bmod 1$ is in an interval $[a,b] \subset [0,1]$, we need to study ${\rm Prob}(Y\bmod 1 \in [a,b]) = \sum_{n=-\infty}^\infty {\rm Prob}(Y \in [a+n, b+n])$. We \emph{sketch} how Poisson summation enters, and provide full details when we prove our main result. The latter probabilities are integrals of the density over the intervals $[a+n, b+n]$, and if $N$ is large each of these is approximately $b-a$ times the density at $n$. By Poisson summation, summing the density over $n$ is the same as summing the Fourier transform at $n$: $$\sum_{n=-\infty}^\infty\frac{1}{\sqrt{N}}\ e^{-\pi
n^2/N} \ = \ \sum_{n=-\infty}^\infty e^{-\pi n^2 N}.$$ Note the sharp contrast between the two sums. For the first sum, all $n$ with $|n| \le \sqrt{N}$ contribute the same order of magnitude, while for the second sum the $n=0$ term contributes 1 and the next term is immensely smaller (by a factor of $e^{-\pi N}$). This example illustrates how Poisson summation allows us to replace a slowly decaying sum of a density with a rapidly decaying one.



\section{Main Results}\label{sec:mainresults}

Our main result is the following extension of results for the exponential distribution, which measures the deviation of the logarithm modulo 1 of Weibulls and the uniform distribution. It's thus not surprising that for $\gamma$ close to 1 the digits are close to Benford, as $\gamma = 1$ corresponds to the exponential distribution. The main contribution below is quantifying how the fit worsens as $\gamma$ grows. The larger $\gamma$ is, the worse the fit. This is intuitively plausible from a plot of the Weibull density; as $\gamma$ increases, the distribution becomes more concentrated near 1. Part of the $\alpha$ is easier to explain. As the effect of replacing $\alpha$ by $\alpha B$ is simply to rescale our random variable by a factor of $B$, the significand is unaffected. Thus it suffices to study $\alpha$ in the window $[1, B)$, but $\gamma$ may be any real value.


\begin{thm} \label{thm:main} Let $Z_{\alpha,\gamma}$ be a random variable whose density is a Weibull with parameters $\alpha, \gamma > 0$ arbitrary. For $z \in
[0,1]$, let $F_B(z)$ be the cumulative distribution function of $\log_BZ_{\alpha,\gamma} \bmod 1$; thus $F_B(z) := {\rm Prob}(\log_BZ_{\alpha,\gamma} \bmod
1 \in [0,z])$. Then the density of $\log_B Z_{\alpha,\gamma} \bmod 1$, $F_B'(z)$, is given by \bea \label{eq:thmeqmainFprime} F_B'(z) &\ = \ & 1 + 2\sum_{m =1}^{ \infty } \Re\left(e^{-2\pi im \left(z -\frac{\log \alpha}{\log B} \right)} \cdot \Gamma \left(1+ \frac{2\pi i m}{\gamma \log B} \right) \right). \eea In particular, the densities of $\log_B Z_{\alpha,\gamma} \bmod 1$ and $\log_B Z_{\alpha B,\gamma}\bmod 1$ are equal, and thus it suffices to consider only $\alpha$ in an interval of the form $[a, aB)$ for any $a>0$. \end{thm}

From the fundamental equivalence, a straightforward integration immediately translates \eqref{eq:thmeqmainFprime} into quantifying differences in the distribution of leading digits of Weibulls and Benford's law. Specifically, the probability of a first digit of $d$ is obtained by integrating $F_B'(z)$ from $\log_B d$ to $\log_B (d+1)$. The main term comes from the constant 1, and is $\log_B \frac{d+1}{d}$, the Benford probability; we discuss the size of the error in Theorem \ref{thm:error}.

The above theorem is proved in the next section. As in \cite{MiNi2}, the proof involves applying Poisson summation to the derivative of the cumulative distribution function of the logarithms modulo 1, which as discussed in the previous section is a natural way to compare deviations from the resulting distribution and the uniform distribution. The key idea is that if a data set satisfies Benford's Law, then the distribution of its logarithms will be uniform. Our series expansions are obtained by applying properties of the Gamma function.

As the deviations of $F_B'(z)$ from being identically 1 measure the deviations from Benford behavior, it is important to have good estimates for the sum over $m$ in \eqref{eq:thmeqmainFprime}. The bounds below have not been optimized, but instead have been chosen to simplify the algebra in the proofs (given in Appendix \ref{sec:boundproofs}). Thus we assume $k$ below is at least 6, which is essentially equivalent to only investigating the case where the error $\epsilon$ is required to be of at most modest size (which is reasonable, as a series expansion with a large error is useless).


\begin{thm}\label{thm:error} Let $F_B'(z)$ be as in \eqref{eq:thmeqmainFprime}.
\begin{enumerate}
\item\label{item:maintwo} For $M \ge \frac{ \gamma \log B \log 2 }{ 4 \pi ^2 }$, the error from dropping the $m \ge M$ terms in $F_B'(z)$ is at most
$$\frac{2\sqrt{2} (\pi^2+\gamma\log B) \sqrt{\gamma \log B}}{\pi^3}\ M\ e^{- \pi ^2 M/ \gamma \log B}.$$

\item\label{item:mainthree} In order to have an error of at most $\epsilon$ in evaluating $F_B'(z)$, it suffices to take the first $M$ terms, where $M = (k + \ln k + 1/2)/a$,
with $k = \max\left(6,- \ln \left(a \epsilon/C \right)\right)$, $a = \pi^2/ (\gamma \log B)$, and  $C =\frac{2\sqrt{2}(\pi^2+\gamma\log B)\sqrt{\gamma \log B}}{\pi^3}$.
\end{enumerate}
\end{thm}

For further analysis, we compared our series expansion for the derivative to the uniform distribution through a Kolmogorov-Smirnov test; see Figure \ref{fig:k-stest} for a contour plot of the discrepancy. This statistic measures the absolute value of the greatest difference in cumulative distribution functions of two densities. Thus the larger the value, the further apart they are. Note the good fit observed between the two distributions when $\gamma =1$ (representing the Exponential distribution), which has already been proven to be a close fit to the Benford distribution (\cite{DL,LSE,MiNi2}).

\begin{figure}[h]
\begin{center}
\includegraphics[scale=.75]{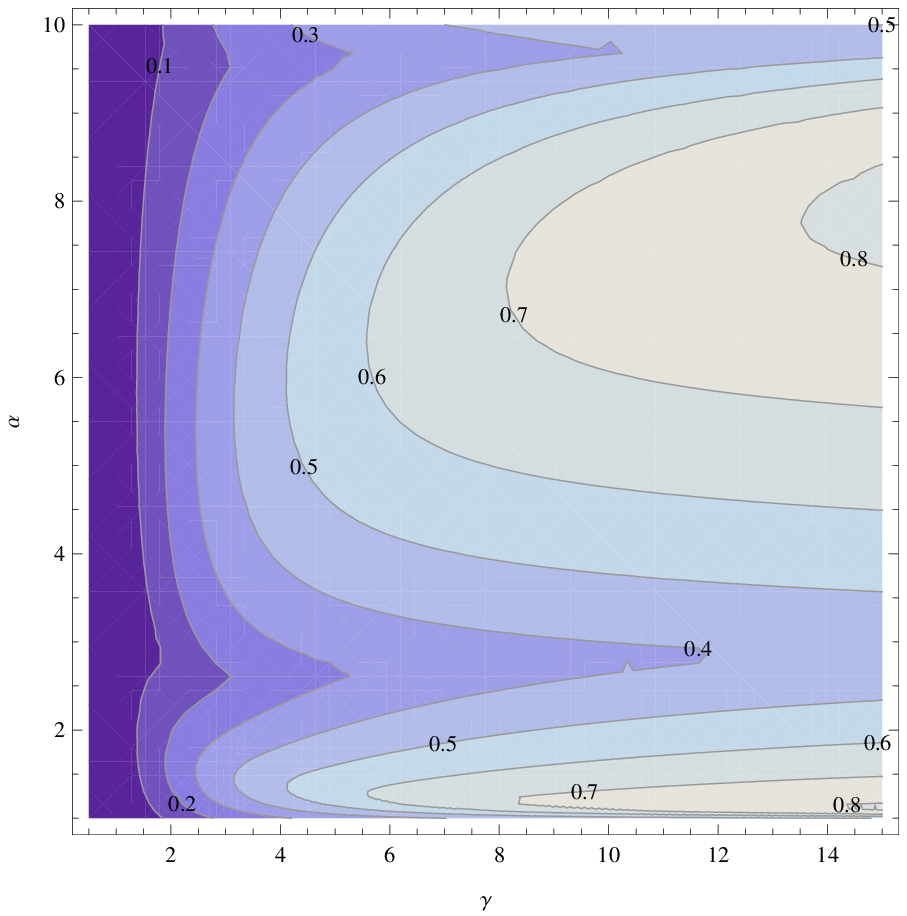}
\includegraphics[scale=.75]{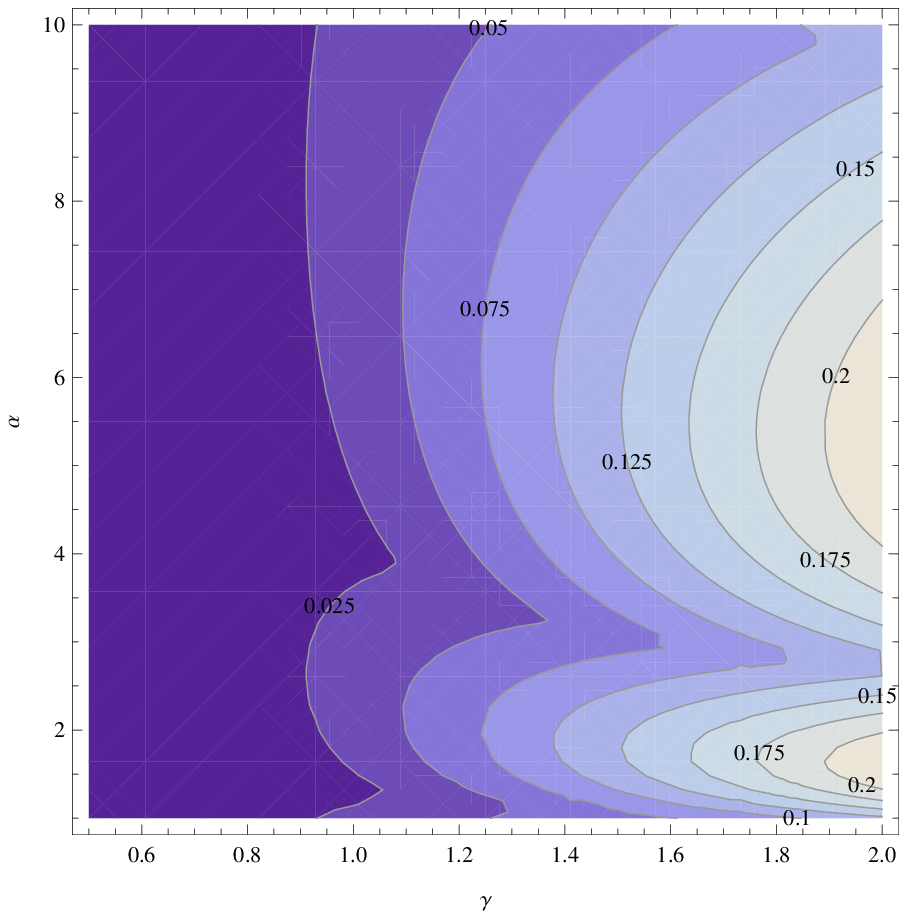}
\caption{Kolmogorov$-$Smirnov Test: Left: $\gamma \in [0, 15]$, Right: $\gamma \in [.5, 2]$. As $\gamma $ (the shape parameter on the x-axis) increases, the Weibull distribution is no longer a good fit compared to the uniform. Note that $\alpha$ (the scale parameter on the y-axis) has less of an effect on the overall conformance. }\label{fig:k-stest}
\end{center} \end{figure}

The Kolmogorov$-$Smirnov metric gives a good comparison because it allows us to compare the distributions in terms of both parameters, $\gamma$ and $\alpha$. We also look at two other measures of closeness, the $L_1$-norm and the $L_2$-norm, both of which also test the differences between  \eqref{eq:thmeqmainFprime} and the uniform distribution; see Figure \ref{fig:l1norm}. The $L_1$-norm of $f-g$ is $\int_0^1 |f(t)-g(t)|dt$, which puts equal weights on the all deviations, while the $L_2$-norm is given by $\int_0^1 |f(t)-g(t)|^2 dt$, which unlike the $L_1$-norm puts more weight on larger differences. The closer $\gamma$ is to zero the better the fit. As $\gamma$ increases the cumulative Weibull distribution is no longer a good fit compared to 1. The $L_1$ and $L_2$-norms are independent of $\alpha$.

\begin{figure}[h]
\begin{center}
\includegraphics[scale=.75]{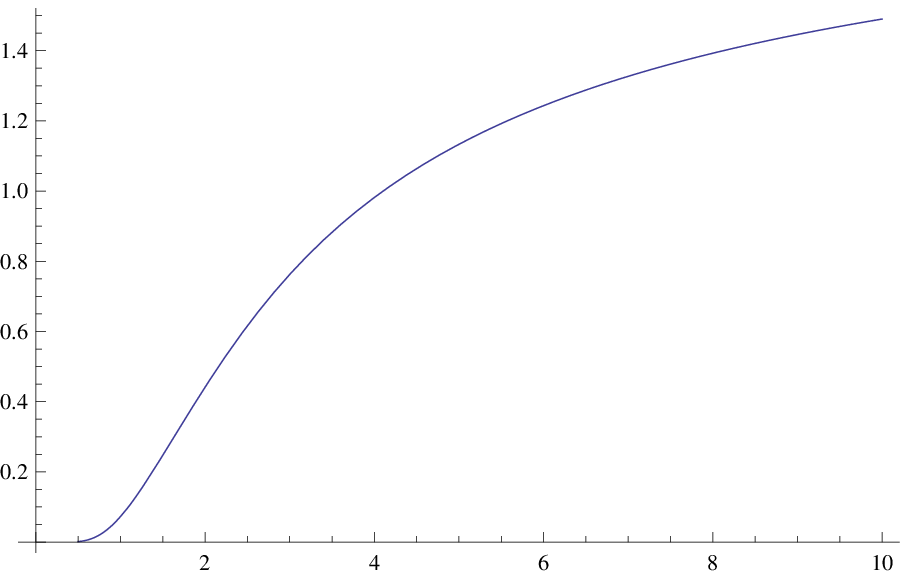}
\includegraphics[scale=.75]{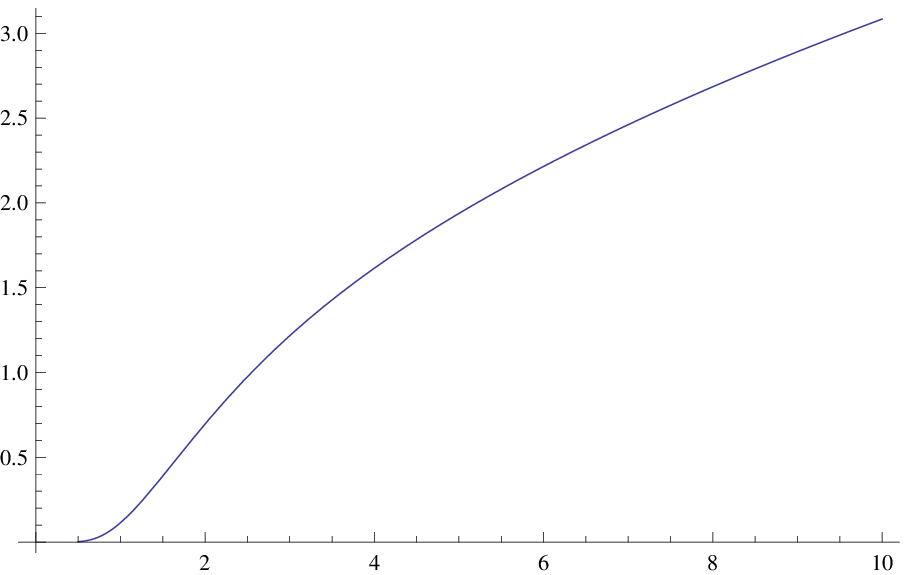}
\caption{Left: $L_1$-norm of $F_B'(z)-1$ for $\gamma \in [0.5, 10]$. Right: $L_2$-norm  of $F_B'(z)-1$ for $\gamma \in [0.5, 10]$.}\label{fig:l1norm}
\end{center} \end{figure}



The combination of the Kolmogorov-Smirnov tests and the $L_1$ and $L_2$ norms show us that the  Weibull distribution almost exhibits Benford behavior when $\gamma$ is modest; as $\gamma$ increases the Weibull no longer conforms to the expected leading digit probabilities. The scale parameter $\alpha$ does have a small effect on the conformance as well, but not nearly to the same extreme as the shape parameter, $\gamma$. Fortunately in many applications the scale parameter $\gamma$ is not too large (it is frequently less than 2 in the Weibull references cited earlier), and thus our work provides additional support for the prevalence of Benford behavior.



\section{Proof of Main Result}\label{sec:approachviaderivs}

To prove Theorem \ref{thm:main}, we study the distribution of $\log_B Z_{\alpha,\gamma} \bmod 1$ when $Z_{\alpha,\gamma}$ has the Weibull distribution with parameters $\alpha$ and $\gamma$. The analysis is aided by the fact that the cumulative distribution function for a Weibull random variable has a nice closed form expression; for $Z_{\alpha,\gamma}$ the cumulative distribution function is $\mathcal{F}_{\alpha,\gamma}(x) = 1 - \exp(-(x/a)^\gamma)$. Let $[a,b] \subset [0,1]$. Then
\bea  \label{eq:keysums}
\mathrm{Prob}(\log_B Z_{\alpha,\gamma} \bmod 1 \in [a,b]) & \ = \ & \sum_{k=-\infty}^\infty \mathrm{Prob}(\log_B Z_{\alpha,\gamma} \bmod 1 \in [a+k,b+k])
\nonumber\\ & = & \sum_{k=-\infty}^\infty {\rm Prob}(Z_{\alpha,\gamma} \in [B^{a+k}, B^{b+k}])
\nonumber\\ &  = & \sum_{k=-\infty}^\infty \left( \exp\left( - \left(\frac{B^{a+k}}{\alpha} \right)^\gamma \right) -  \exp\left( - \left(\frac{B^{b+k}}{\alpha} \right)^\gamma \right)\right).\nonumber\\
\eea

\begin{proof}[Proof of Theorem \ref{thm:main}] It suffices to investigate \eqref{eq:keysums} in the special case when $a=0$ and $b=z$, since for any other interval $[a,b]$ we may determine its probability by subtracting the probability of $[0,a]$ from $[0,b]$. Thus, we study the cumulative distribution function of $\log_B Z_{\alpha,\gamma} \bmod 1$ for $z \in  [0,1]$, which we denote by $F_B(z)$:
\bea \label{probfunc1} F_B(z) & \ := \ & \mathrm{Prob}\left(\log_B Z_{\alpha,\gamma} \bmod 1 \in [0,z]\right)\nonumber\\ & \ = \ & \sum_{k=-\infty}^\infty \left( \exp\left( - \left(\frac{B^{k}}{\alpha} \right)^\gamma \right) -  \exp\left( - \left(\frac{B^{z+k}}{\alpha} \right)^\gamma \right)\right).
\eea
This series expansion is rapidly converging, and the closeness of $Z_{\alpha,\gamma}$ to Benford is equivalent to the rapidly converging series in \eqref{probfunc1} for $F_B(z)$ being close to $z$ for all $z$.

A natural way to investigate the closeness of $F_B(z)$ to $z$ is to compare $F'(z)$ to 1. As in \cite{MiNi2}, studying the derivative $F'_B(z)$ is an easier way to approach this problem, because we obtain a simpler Fourier transform than the Fourier transform of $e^{ - \left(\frac{B^{k}}{\alpha} \right)^\gamma} -  e^{- \left(\frac{B^{z+k}}{\alpha} \right)^\gamma} $. We then can analyze the obtained Fourier transform by applying Poisson summation (Theorem \ref{thm:poissonsum}).

We use the fact that the derivative of the infinite sum $F_B(z)$ is the sum of the derivatives of the individual summands. This is justified by the rapid decay of summands, yielding
\bea \label{derivative1} F_B'(z) & \ = \ &  \sum_{k=-\infty}^\infty \frac{1}{\alpha} \cdot \left[ \exp\left( - \left( \frac{B^{z+k}}{\alpha} \right)^\gamma \right)B^{z+k}\left(\frac{B^{z+k}}{\alpha}\right)^{\gamma -1}\gamma\log B \right]
\nonumber\\ & = & \sum_{k=-\infty}^\infty \frac{1}{\alpha} \cdot \left[ \exp\left( - \left( \frac{\zeta B^{k}}{\alpha} \right)^\gamma \right)\zeta B^{k}\left(\frac{\zeta B^{k}}{\alpha}\right)^{\gamma -1}\gamma\log B \right],
\eea where for $z \in [0,1]$, we use the change of variables $\zeta = B^z$.

We introduce $H(t) = \frac{1}{\alpha} \cdot \exp\left( - \left( \frac{\zeta B^{t}}{\alpha} \right)^\gamma \right)\zeta B^{t}\left(\frac{\zeta B^{t}}{\alpha}\right)^{\gamma -1}\gamma\log B$, where $\zeta \geq 1$ as $\zeta = B^z$ with $z \ge 0$. Since $H(t)$ is decaying rapidly we may apply Poisson summation, thus
\bea \label{poisson} \sum_{k=-\infty}^\infty H(k) \ = \ \sum_{k=-\infty}^\infty \widehat{H}(k), \eea
\noindent where $ \widehat{H}$ is the Fourier Transform of $H: \widehat{H}(u)= \int_{-\infty}^\infty H(t)e^{-2 \pi itu} dt$. Therefore
\bea \label{fouriertrans} F'_B(z) & = & \sum_{k=-\infty}^\infty H(k) \ = \  \sum_{k=-\infty}^\infty \widehat{H}(k)
\nonumber\\ & = & \sum_{k=-\infty}^\infty \int_{-\infty}^\infty \frac{1}{\alpha} \cdot \exp\left( - \left( \frac{\zeta B^{t}}{\alpha} \right)^\gamma \right)\zeta B^{t}\left(\frac{\zeta B^{t}}{\alpha}\right)^{\gamma -1}\gamma\log B \cdot e^{-2 \pi itk} dt.
\eea

We change variables again, setting $w= \left(\zeta B^t/\alpha\right)^\gamma$, which implies \be t\ =\ \log_B \left( \frac{\alpha w^{1/\gamma}}{\zeta} \right) \ \ {\rm and}\ \  dw\ =\ \frac{1}{\alpha} \left( \frac{\zeta B^t}{\alpha} \right)^{\gamma -1} \cdot \zeta B^t \gamma \log B\ dt, \ee so that
\bea \label{fourierchange} F'_B(z) & = & \sum_{k=-\infty}^\infty \int_{0}^\infty e^{-w} \cdot \exp\left( -2\pi ik\cdot \log_B \left( \frac{\alpha w^{1/\gamma}}{\zeta} \right) \right) dw
\nonumber\\ & = & \sum_{k=-\infty}^\infty \left( \frac{\alpha}{\zeta} \right)^{-2\pi ik/ \log B} \int_{0}^\infty e^{-w} \cdot w^{-2\pi ik/\gamma \log B} dw
\nonumber\\ & = & \sum_{k=-\infty}^\infty \left( \frac{\alpha}{\zeta} \right)^{-2\pi ik/ \log B} \Gamma\left(1- \frac{2\pi ik}{\gamma \log B}\right),
\eea
where we used the definition of the $\Gamma$-function in the last line. As $\Gamma(1)=1$, we have
\bea \label{gamma1} F'_B(z) \ = \ 1 + \sum_{m=1}^\infty \left[ \left( \frac{\zeta}{\alpha} \right)^{ \frac{2\pi im}{\log B} } \Gamma \left( 1 - \frac{2\pi im}{\gamma \log B} \right) + \left( \frac{\zeta}{\alpha} \right)^{ \frac{-2\pi im}{\log B}  } \Gamma \left( 1 + \frac{2\pi im}{\gamma \log B} \right) \right]
\eea
As in \cite{MiNi2}, the above series expansion is rapidly convergent. As $\zeta=B^z$ we have
\bea \left( \frac{\zeta}{\alpha} \right)^{2\pi im/ \log B} = \cos \left(2\pi mz -2\pi m \left( \frac{\log \alpha}{\log B} \right) \right) + i \sin \left(2\pi mz -2\pi m \left( \frac{\log \alpha}{\log B} \right) \right), \eea
which gives a Fourier series expansion for $F_B'(z)$ with coefficients arising from special values of the $\Gamma$-function.

Using properties of the $\Gamma$-function we are able to improve (\ref{gamma1}). If $y\in \R$ then $\Gamma(1-iy) =
\overline{\Gamma(1+iy)}$ (where the bar denotes complex
conjugation). Thus the $m$\textsuperscript{th} summand in
(\ref{gamma1}) is the sum of a number and its complex
conjugate, which is simply twice the real part. We use the following standard relationship (see for example \cite{AS}):
\bea \label{gamma1+ix} |\Gamma(1+ix)|^2 \ = \ \frac{\pi
x}{\sinh(\pi x)} \ = \ \frac{2 \pi x}{e^{\pi x} - e^{-\pi x}}. \eea

Writing the summands in (\ref{gamma1}) as $2\Re\left(e^{-2\pi im \left(z -\frac{\log \alpha}{\log B} \right)} \cdot \Gamma \left(1+ \frac{2\pi i m}{\gamma \log B} \right) \right)$, (\ref{gamma1}) becomes
\bea \label{Fprimebgammanext}
F_B'(z)  & = & 1 + 2\sum_{m =
1}^{\infty} \Re\left(e^{-2\pi im \left(z -\frac{\log \alpha}{\log B} \right)} \cdot \Gamma \left(1+ \frac{2\pi i m}{\gamma \log B} \right) \right).   \eea
Finally, in the exponential argument above there is no change in replacing $\alpha$ with $\alpha B$, as this changes the argument by $2\pi i$. Thus it suffices to consider $\alpha \in [a, aB)$ for any $a>0$.
\end{proof}


This proof demonstrates the power of using Poisson summation in Benford's law problems, as it allows us to convert a slowly convergent series expansion into a rapidly converging one, with the main term corresponding to Benford behavior and the other terms measuring the deviation.

\appendix

\section{Proofs of Bounding Estimates}\label{sec:boundproofs}

We first estimate the contribution to $F_B'(z)$ from the tail, say from the terms with $m \ge M$. We do not attempt to derive the sharpest bounds possible, but rather highlight the method in a general enough case to provide useful estimates.

\begin{proof}[Proof of Theorem \ref{thm:error}\eqref{item:maintwo}] We must bound the truncation error
\bea \mathcal{E}_B(z) \ := \
\Re \sum_{m=M}^\infty e^{-2 \pi i m \left( z -\frac{\log \alpha}{\log B} \right)} \cdot \Gamma \left( 1 + \frac{2 \pi i m}{\gamma \log B} \right),
\eea where $ \Gamma(1+i u)= \int_0^\infty e^{-x} x^{i u} dx = \int_0^\infty e^{-x} e^{i u \log x} dx$. Note that in our case, $u = \frac{2 \pi m}{\gamma \log B}$. As $u$ increases there is more oscillation and therefore more cancelation, resulting in a smaller value for our integral. Since $|e^{i \theta}| = 1$, if we take absolute values inside the sum we have $|e^{-2 \pi i m \left(z -\frac{\log \alpha}{\log B}\right)}| = 1$, and thus we may ignore this term in computing an upper bound.

Using standard properties of the Gamma function, we have
\bea \label{eq:error1} |\Gamma(1+i x)|^2 &\ =\ & \frac{\pi x}{\sinh (\pi x)}\ =\ \frac{2 \pi x}{e^{\pi x} - e^{-\pi x}}, \ \ \  \mathrm{where} \   x = \frac{2 \pi m}{\gamma \log B}. \eea This yields
\bea |\mathcal{E}_B(z)| &\ \leq\ & \sum_{m=M}^\infty 1 \cdot \left( \frac{4 \pi ^2 m}{\gamma \log B} \cdot \frac{1}{e^{2 \pi ^2 m / \gamma \log B} -e^{-2 \pi ^2 m / \gamma \log B} } \right)^{1/2}.
\eea

Let $u = e^{2 \pi ^2 m / \gamma \log B}$. We overestimate our error term by removing the difference of the exponentials in the denominator. Simple algebra shows that for $\frac{1}{u - \frac{1}{u} } \le \frac{2}{u}$ we need $u \ge \sqrt{2}$. For us this means
$e^{2 \pi ^2 m / \gamma \log B} \ge \sqrt{2}$, allowing us to simplify the denominator if $m \ge \frac{ \gamma \log B \log 2 }{ 4 \pi ^2 }$, which we may do as we assumed $M$ exceeds this value and $m \ge M$. We substitute this bound into \eqref{eq:error1}, and replace $\sqrt{m}$ with $m$ to simplify the resulting integral:
\bea
|\mathcal{E}_B(z)| \ \leq \ \sum_{m=M}^\infty \left( \frac{4 \pi ^2 m}{\gamma \log B} \right)^{1/2} \cdot \frac{\sqrt{2}}{e^{\pi ^2 m / \gamma \log B}} \ \leq \ \frac{2 \sqrt{2} \pi}{\sqrt{\gamma \log B}} \int_M^\infty m e^{- \pi ^2 m / \gamma \log B} dm.
\eea Letting $a = \pi^2/\gamma\log B$, integrating by parts gives
\bea \label{eq:error2}
|\mathcal{E}_B(z)|& \  \leq \ & \frac{2 \sqrt{2} \pi}{\sqrt{\gamma \log B}} \frac1{a^2}\left(aM e^{-aM} + e^{-aM}\right) \ \le \ \frac{2 \sqrt{2} \pi}{\sqrt{\gamma \log B}} \frac{a+1}{a^2} M e^{-aM} \eea (since $M\ge 1$, $aM+1 \le (a+1)M$), which after some algebra simplifies to \bea |\mathcal{E}_B(z)|& \  \leq \ & \frac{2\sqrt{2} (\pi^2+\gamma\log B) \sqrt{\gamma \log B}}{\pi^3}\ M\ e^{- \pi ^2 M/ \gamma \log B}, \eea
which is the error listed in Theorem \ref{thm:error}(1). \end{proof}

\ \\

\begin{proof}[Proof of Theorem \ref{thm:error}\eqref{item:mainthree}] Given the estimation of the error term from above, we now ask the related question of, given an $\epsilon > 0$, how large must $M$ be so that the first $M$ terms give $F_B'(z)$ accurately to within $\epsilon$ of the true value. Let $C = \frac{2\sqrt{2}(\pi^2+\gamma\log B)\sqrt{\gamma \log B}}{\pi^3}$ and $a = \frac{\pi ^2}{\gamma \log B}$. We must choose $M$ so that $C M e^{-aM} \leq \epsilon$, or equivalently
\bea \frac{C}{a}\ a  M e^{-aM}\ \leq\ \epsilon. \eea
As this is a transcendental equation in $M$, we do not expect a nice closed form solution, but we can obtain a closed form expression for a bound on $M$; for any specific choices of $C$  and $a$ we can easily numerically approximate $M$. We let $u = aM$, giving
\bea
u e^{-u}\ \leq\ a\epsilon/C. \eea
With a further change of variables, we let $k = - \ln (a\epsilon/C)$ and then expand $u$ as $u=k+x$ (as the solution should be close to $k$). We find
\bea
u \cdot e^{-u} \ \leq\  e^{-k}  \ \ \ {\rm is\ equivalent\ to} \ \ \ \frac{k + x}{e^x} \ \leq\  1. \eea
We try $ x = \ln k + \frac{1}{2}$, and see
\bea
\frac{k + x}{e^x} \ \leq\   1  \ \ \ {\rm is\ equivalent\ to} \ \ \
\frac{k + \ln k + \frac{1}{2}}{ k \cdot e^{1/2}} \ \leq \ 1. \eea
From here, we want to determine the value of $k$ such that $\ln k \leq \frac{1}{2} k$, as this ensures the needed inequality above holds. Exponentiating, we need $k^2 \le e^k$. As $e^k \ge k^3/3!$ for $k$ positive, it suffices to choose $k$ so that $k^2 \le k^3/6$, or $k \ge 6$; this holds for $\epsilon$ sufficiently small.
For $k \geq 6$, we have  \be k + \ln k + \frac{1}{2}\ \leq\ k + \frac{1}{2} k + \frac{1}{12} k\ =\ \frac{19}{12} k\ \approx\ 1.5833k, \ee but  \be k \cdot e^{1/2}\ \approx\ 1.64872 k.\ee Therefore a correct cutoff value for $M$, in order to have an error of at most $\epsilon$, is
\bea M &\ =\ & \ \frac{k + \ln k + \frac{1}{2}}{a}, \eea
where \bea k \ = \ \max\left(k,- \ln \left(\frac{a \epsilon }{C} \right)\right),\ \ \ a \ = \ \frac{\pi ^2}{\gamma \log B},\ \ \  C \ = \ \frac{2\sqrt{2}(\pi^2+\gamma\log B)\sqrt{\gamma \log B}}{\pi^3}.  \eea \end{proof}

\ \\


\begin{thebibliography}{AHMP-GQq}


\bibitem[AHMP-GQ]{AHMP-GQ}
C. T. Abdallah, Gregory L. Heileman,  S. J. Miller, F. P\'erez-Gonz\'alez and T. Quach, \emph{Application of Benford's Law to Images}, in Theory and Applications of Benford's Law (Steven J. Miller, editor), Princeton University Press, to appear.

\bibitem[AS]{AS}
M. Abromovich and I. A. Stegun, \emph{Handbook of mathematical functions with formulas, graphs, and mathematical tables}, tenth printing, National Burea of Standards, Applied Mathematics Series \textbf{55}, 1972. Available online at \burl{http://people.math.sfu.ca/~cbm/aands/}.

\bibitem[AP]{AP}
D. Aldous and T. Phan, \emph{When Can One Test an Explanation? Compare and Contrast Benford's Law and the Fuzzy CLT}, The American Statistician \textbf{64} (2010), no. 3, 221--227.

\bibitem[An]{An}
Anonymous, \emph{Weibull Survival Probability Distribution}, last modified May 11, 2009. \scriptsize \burl{http://www.wepapers.com/Papers/30999/Weibull_Survival_Probability_Distribution}. \footnotesize

\bibitem[AJ]{AJ}
L. Arshadi and A. H. Jahangir, \emph{Benford's law behavior of Internet traffic}, Journal of Network and Computer Applications (2013), in press.
\burl{http://dx.doi.org/10.1016/j.jnca.2013.09.007i}.


\bibitem[Ben]{Ben}
F. Benford, \emph{The Law of Anomalous Numbers}, Proceedings of the American Philosophical Society \textbf{78} (1938), 551-572.

\bibitem[BH1]{BH1}
A. Berger and T. P. Hill, \emph{Fundamental Flaws in Feller's Classical Derivation of Benford's Law}, preprint 2010. \burl{http://arxiv.org/abs/1005.2598}.

\bibitem[BH2]{BH2}
A. Berger and T. P. Hill, \emph{Benford's Law Strikes Back: No Simple Explanation in Sight for Mathematical Gem}, The Mathematical Intelligencer \textbf{3} (2011), no. 1, 85--91.

\bibitem[BH3]{BH3}
A. Berger and T. P. Hill, \emph{A basic theory of Benford's Law}, Probability Surveys \textbf{8} (2011), 1--126.

\bibitem[BH4]{BH4}
A. Berger and T. P. Hill, \emph{Benford Online Bibliography}, \burl{http://www.benfordonline.net/}, 2012.

\bibitem[BS]{BS10}
H. W. Block and T. H. Savits, \emph{A General Example for Benford Data}, The American Statistician \textbf{64} (2010), no. 4, 335--339.

\bibitem[Ca]{Ca}
K. J. Carroll, \emph{On the use and utility of the Weibull model in the analysis of survival data}, Controlled Clinical Trials \textbf{24} (2003), no. 6, 682--701.

\bibitem[CG]{CG07}
W. K. T. Cho and B. J. Gaines, \emph{Breaking the (Benford) Law: Statistical Fraud Detection in Campaign Finance}, The American Statistician \textbf{67} (2007), no. 3, 218--223.

\bibitem[CB]{CB}
O. Corzoa and N. Brachob, \emph{Application of Weibull distribution model to describe the vacuum pulse osmotic dehydration of sardine sheets},
LWT - Food Science and Technology \textbf{41} (2008), no. 6, 1108--1115.

\bibitem[Cr]{Cr}
T. S. Creasy, \emph{A method of extracting Weibull survival model parameters from filament bundle load/strain data}, Composites Science and Technology \textbf{60} (2000), no. 6,  825--832.


\bibitem[DL]{DL}
L. Dumbgen and C. Leuenberger,
\emph{Explicit bounds for the approximation error in Benford's law},
Electronic Communications in Probability \textbf{13} (2008), 99--112.

\bibitem[Dia]{Dia} \newblock P. Diaconis, \emph{The distribution of leading
digits and uniform distribution mod 1}, Ann. Probab. \textbf{5}
(1979), 72--81.

\bibitem[Fel]{Fe}
W. Feller, \emph{An Introduction to Probability Theory and Its
Applications}, 2nd edition, Vol. II, John Wiley \& Sons, New York,
1966.

\bibitem[Few]{Few}
R. M. Fewster, \emph{A simple explanation of Benford's Law}, The American Statistician \textbf{63} (2009), no. 1, 26--32.

\bibitem[Fr]{Fr}
S. Fry, \emph{How political rhetoric contributes to the stability of coercive rule: A Weibull model of post-abuse government survival}. Paper presented at the annual meeting of the International Studies Association, Le Centre Sheraton Hotel, Montreal, Quebec, Canada, Mar 17, 2004.

\bibitem[Hi1]{Hi2}
T. P. Hill, \emph{A Statistical Derivation of the Significant-Digit Law}, Statistical Science \textbf{10} (1995), no. 4, 354-363.

\bibitem[Hi2]{Hi1} \newblock T. P. Hill, \emph{The first-digit phenomenon}, American Scientists \textbf{86}  (1996), 358--363.

\bibitem[Hi3]{Hi3}
T. P. Hill, \emph{Benford's Law Blunders}, Letters to the Editor, The American Statistitican \textbf{65} (2011), no. 2, 141--141.

\bibitem[JKKKM]{JKKKM}
D. Jang, J. U. Kang, A. Kruckman, J. Kudo and S. J. Miller, \emph{Chains of distributions, hierarchical Bayesian models and Benford's Law}, Journal of Algebra, Number Theory: Advances and Applications, volume 1, number 1 (March 2009), 37--60.

\bibitem[JS]{JS}
G. Judge and L. Schechter, \emph{Detecting problems in survey data using Benford's law}, Journal of Human Resources \textbf{44} (2009), no. 1, 1--24.

\bibitem[Knu]{Knu} \newblock D. Knuth, \emph{The Art of Computer
Programming, Volume 2: Seminumerical Algorithms}, Addison--Wesley,
third edition, 1997. 

\bibitem[KM]{KM}
A. Kontorovich and S. J. Miller, \emph{Benford's Law, values of $L$-functions and the $3x+1$
problem}, Acta Arithmetica \textbf{120} (2005), no. 3, 269--297.

\bibitem[LSE]{LSE}
L. M. Leemis, B. W. Schmeiser and D. L. Evans, \emph{Survival Distributions Satisfying Benford's Law}, The American Statistician \textbf{54} (2000), no. 3.

\bibitem[MABF]{MABF}
B. McShane, M. Adrian, E. T. Bradlow  and P. S. Fader, \emph{Count Models Based on
Weibull Interarrival Times}, Journal of Business and Economic Statistics \textbf{26} (2006), no. 3, 369--378.

\bibitem[Me]{Me}
W. Mebane, \emph{Election Forensics: The Second-Digit Benford's Law Test and Recent American Presidential Elections}, Election Fraud Conference, Salt Lake city, Utah, September 29-30, 2006. \burl{http://www.umich.edu/~wmebane/fraud06.pdf}.

\bibitem[Mik]{Mik}
P. G. Mikolaj, \emph{Environmental Applications of the Weibull Distribution Function: Oil Pollution},
Science 2 \textbf{176} (1972), no. 4038, 1019--1021.

\bibitem[Mi]{Mi}
S. J. Miller, \emph{A derivation of the Pythagorean Won-Loss Formula in baseball},
Chance Magazine \textbf{20} (2007), no. 1, 40--48.

\bibitem[MiNi1]{MiNi1}
\newblock S. J. Miller and M. Nigrini, \emph{The Modulo $1$ Central Limit Theorem and Benford's Law for Products}, International Journal of Algebra \textbf{2} (2008), no. 3, 119--130.

\bibitem[MiNi2]{MiNi2}
S. J. Miller and M. J. Nigrini, \emph{Order Statistics and Benford's Law}, International Journal of Mathematics and Mathematical Sciences, (2008), 1-13.

\bibitem[MT-B]{MT-B}
\newblock S. J. Miller and R. Takloo-Bighash, \emph{An Invitation to Modern Number Theory}, Princeton University Press, Princeton, NJ, 2006.

\bibitem[Ne]{Ne} \newblock S. Newcomb, \emph{Note on the frequency of use
of the different digits in natural numbers}, Amer. J. Math.
\textbf{4} (1881), 39-40.

\bibitem[Nig1]{Nig1}
M. J. Nigrini, \emph{Digital Analysis and the Reduction of Auditor Litigation Risk}. Pages 68-81 in \emph{Proceedings of the 1996 Deloitte \& Touche / University of Kansas Symposium on Auditing Problems}, ed. M. Ettredge, University of Kansas, Lawrence, KS, 1996.

\bibitem[Nig2]{Nig2}
M. J. Nigrini, \emph{The Use of Benford's Law as an Aid in Analytical Procedures}, Auditing: A Journal of Practice \& Theory, \textbf{16} (1997), no. 2, 52-67.

\bibitem[NiMi]{NiMi}
M. J. Nigrini and S. J. Miller, \emph{Data diagnostics using second order tests of Benford's Law}, Auditing: A Journal of Practice and Theory \textbf{28} (2009), no. 2, 305--324.

\bibitem[Rai]{Rai}
R. A. Raimi, \emph{The First Digit Problem}, The American Mathematical Monthly, \textbf{83:7}  (1976), no. 7, 521-538.

\bibitem[SS]{SS}
E. Stein and R. Shakarchi, \emph{Fourier Analysis: An
Introduction}, Princeton University Press, Princeton, NJ, 2003.

\bibitem[TKD]{TKD}
Y. Terawaki, T. Katsumi and V. Ducrocq, \emph{Development of a survival model with piecewise Weibull baselines for the analysis of length of productive life of Holstein cows in Japan}, Journal of Dairy Science  \textbf{89} (2006), no. 10, 4058--4065.

\bibitem[We]{We}
W. Weibull, \emph{A statistical distribution function of wide applicability}, J. Appl. Mech. \textbf{18} (1951), 293--297.

\bibitem[WW]{WW}
E. Whittaker and G. Watson, \emph{A Course of Modern
Analysis}, 4th edition, Cambridge University Press, Cambridge,
1996.


\bibitem[Yi]{Yi}
C. T. Yiannoutsos, \emph{Modeling AIDS survival after initiation of antiretroviral treatment by Weibull models with changepoints}, J Int AIDS Soc. \textbf{12} (2009), no. 9, doi:10.1186/1758-2652-12-9.

\bibitem[ZLYM]{ZLYM}
Y. Zhao, A. H. Lee, K. K. W. Yau, and G. J. McLachlan, \emph{Assessing the adequacy of Weibull survival models: a simulated envelope approach}, Journal of Applied Statistics (2011), to appear.

\end{thebibliography}
\end{document}